\documentclass{amsart}

\usepackage{xspace}
\usepackage{amssymb}
\usepackage{amsmath,amscd}
\usepackage{amsfonts}
\usepackage{amsthm}
\usepackage{keyval}
\usepackage{array}
\usepackage[all]{xy}
\usepackage{setspace}
\usepackage{enumerate}

\title{A modular description of $ER(2)$} 

\author{Romie Banerjee}  
\address{School of Mathematics, Tata Institute of Fundamental Research, Mumbai 400005} 
\email{banerjee@math.tifr.res.in}   


\begin{document}

\begin{abstract}  
We give a description of the maximally unramified extension of completed second Real Johnson-Wilson theory using supersingular elliptic curves with $\Gamma_0(3)$-level structures.
\end{abstract} 

\maketitle


\newtheorem{prop}{Proposition}[section]
\newtheorem{lemma}{Lemma}[section]
\newtheorem{theorem}{Theorem}[section]
\newtheorem{definition}{Definition}[section]
\newtheorem{cor}{Corollary}[section]
\newtheorem{remark}{Remark}[section]

\newcommand{\quo}{\big/\!\!\big/}

\section{Introduction}

\subsection{Real Johnson-Wilson theories}

Complex conjugation acts on the complex $K$-theory spectrum $KU$ and the homotopy fixed points of this action is $KO$. In fact the complex orientation $MU \rightarrow KU$ is equivariant with respect to this $C_2$-action. Localizing at the prime $2$, there is a $C_2$-equivariant orientation $$BP \rightarrow KU_{(2)}$$ with kernel $\langle v_i, i > 1 \rangle$.  Hu and Kriz (\cite{Hu1},\cite{HK}) have realized this as a map of honest $C_2$-equivariant spectra $$BP\mathbb{R} \rightarrow K\mathbb{R}$$ where $K\mathbb{R}$ is Atiyah's {\it real} $K$-theory (\cite{Atiyah}) localized at the prime $2$. They have further generalized this to a map of equivariant spectra $$BP\mathbb{R} \rightarrow E\mathbb{R}(n)$$ which is an equivariant refinement of the orientation $BP \rightarrow E(n)$ with kernel $\langle v_i , i>n \rangle$. Here $E(n)$ is the $2(2^n-1)$-periodic Johnson-Wilson theory. It has coefficients $\mathbb{Z}_{(2)}[v_1,\ldots,v_n^{\pm1}]$ with $|v_i|=2(2^i-1)$. 

The underlying nonequivariant spectrum of $E\mathbb{R}(n)$ is $E(n)$ and the homotopy fixed points with respect to the complex conjugation action is denoted by $ER(n)$. This is $2^{n+2}(2^n-1)$ periodic. The spectrum $ER(1)$ is $KO_{(2)}$. Kitchloo and Wilson have done extensive computations with the spectrum $ER(2)$ (\cite{KW}).

The spectrum $ER(n)$ is homotopy commutative, but its completion is an $E_{\infty}$-ring spectrum. In fact Averett (\cite{MA}) has shown that after completion the fixed points inclusion map $ER(n) \rightarrow E(n)$ is a higher chromatic generalization of the $C_2$-Galois extension (in the sense of \cite{Rognes}) $KO \rightarrow KU$.


More elaborately, let $E_n$ be the $n$-th Morava $E$-theory spectrum associated to the Lubin-Tate space of deformations of the Honda formal group $\Gamma_n$ over $\mathbb{F}_{2^n}$. This Lubin-Tate space is noncanonically isomorphic to $\hbox{Spf}\,\pi_0(E_n) =\hbox{Spf}\,W(\mathbb{F}_{2^n})[[u_1,\ldots,u_{n-1}]]$ where $W(k)$ denotes the Witt ring of $k$. Hopkins-Miller theory gives a unique $E_{\infty}$-ring structure on $E_n$, with coefficients $$\pi_*E_n = \pi_0(E_n)[u^{\pm 1}]$$ where $|u|=2$. The $n$-th Morava stabilizer group $\mathbb{S}_n = \mathcal{O}^{\times}_{D_{\frac{1}{n}},\mathbb{Q}_2}$
(the group of units in the maximal order of the division algebra over $\mathbb{Q}_2$ with Hasse invariant $\frac{1}{n}$) and the Galois group $Gal= \hbox{Gal}(\mathbb{F}_{2^n}/\mathbb{F}_2)$ act on the Lubin-Tate space. By Hopkins-Miller theory this lifts to an action of the extended stabilizer group $\mathbb{S}_n \rtimes Gal$ on the spectrum $E_n$ through $E_{\infty}$-ring maps, upto contractible choices.

There is a subgroup $H(n) = \mathbb{F}_{2^n}^{\times} \rtimes Gal$ of the extended stabilizer and a central element $[-1]_{\Gamma_n} \in \mathbb{S}_n$ corresponding to the formal inverse of the Honda formal group law. Averett shows that the standard equivalence $L_{K(n)}E(n) \simeq E_n^{hH(n)}$ is equivariant with respect to the $C_2$ action coming from $[-1]_{\Gamma_n}$. In other words there is an equivalence 
\begin{equation}\label{Averett}
L_{K(n)}ER(n) \simeq E_n^{h(C_2 \times H(n))}.
\end{equation}

\subsection{Main result}
In this paper we want to show that the Real Johnson-Wilson theory $ER(2)$ arises from certain modular curves in the same way the spectrum $TMF$ of topological modular forms arises from the moduli stack of smooth elliptic curves \cite{HM}.

The stabilizer group $\mathbb{S}_2$ has the maximal finite subgroup $\widehat{A}_4= Q_8 \rtimes C_3$ where $C_3$ acts on $Q_8$ by conjugation, and $H(2) = C_3 \rtimes Gal \subset \widehat{A}_4 \rtimes Gal$.

Completion of $TMF$ at the chromatic height $2$ admits the description (for the primes $2$ and $3$)

\begin{equation}
L_{K(2)}TMF = \left( \prod_{x \in {Ell}^{ss}(\overline{\mathbb{F}}_p)} E_2^{h\,Aut(x)} \right) ^{h\,{Gal}(\overline{\mathbb{F}}_p/\mathbb{F}_p)}.
\end{equation}

Here ${Ell}^{ss}$ is the locus of supersingular elliptic curves at the prime $p$ and $E_2$ is the Hopkins-Miller spectrum associated to the Lubin-Tate space of deformations of the formal group law of $x$  over $\overline{\mathbb{F}}_p$. Since there is a unique isomorphism class of elliptic curves $x$ at these primes and $Aut(x)$ is the maximal finite subgroup of $\mathbb{S}_2$, the right hand side is the higher real $K$-theory $EO_2^{h\,Gal(\overline{\mathbb{F}}_2/\mathbb{F}_2)}$. 

We want to describe completed $ER(2)$ in terms of modular curves. This raises the question: {\it Given $p=2$, does there exist an \'etale cover of $\mathcal{X}$ of $\mathsf{Ell}^{ss}$ so that there exists $x \in \mathcal{X}$ for which $\hbox{Aut}(x)$ is the subgroup $C_2 \times C_3$ of $\mathbb{S}_2$?}
 

We state the main theorem here.

\begin{theorem}\label{main}
For a $p$-local commutative $S$-algebra $A$, let $A(\mu_{\infty,p})$ denote the $S$-algebra obtained by adjoining all the roots of unity of order prime to $p$. Then $L_{K(2)}ER(2)(\mu_{\infty,2})$ is an algebra over $TMF$.

There is a $K(2)$-local $C_2$-Galois extension $$L_{K(2)}TMF_0(3)(\mu_{\infty,2}) \rightarrow L_{K(2)}TMF_1(3)(\mu_{\infty,2})$$ which is isomorphic to the extension $L_{K(2)}ER(2)(\mu_{\infty,2}) \rightarrow L_{K(2)}E(2)(\mu_{\infty,2})$.
\end{theorem}

It is worth pointing out that Behrens and Hopkins (\cite{BH}) have answered in great detail the following question: Given a {\it maximal} finite subgroup $G$ of $\mathbb{S}_n$ at the prime $p$, is the associated higher real $K$-theory $EO_n$ a summand of the $K(n)$-localization of a $TAF$ spectrum associated to a unitary similitude group of type $(1,n-1)$? This paper is concerned with a similar question where $p=2$, $n=2$ and $G$ is not maximal.

It is evident from the calculations of Mahowld and Rezk (\cite{MR}) that the spectrum $TMF_1(3)$ is an example of a {\it generalized} $E(2)$ (see \cite{Baker}). Presumably there is a notion of a generalized $E\mathbb{R}(n)$ so that $TMF_0(3)$ is an example of a generalized $ER(2)$.

In the next section we reformulate our main theorem  (as Theorem \ref{restate}) in algebraic-geometric language and review the various moduli spaces appearing in the formulation. The final section contains the proof of the main theorem.

{\bf Acknowledgements}. The author has benefited greatly from discussions with and comments received from Jack Morava, Niko Naumann and Andrew Salch.

\section{Setup}
\subsection{Stacks from ring spectra}
Let $BP$ denote the $p$-local Brown-Peterson spectrum. We can consider the associated flat Hopf algebroid. $$V= BP_*\simeq\mathbb{Z}_{(p)}[v_1,\ldots]$$ $$W = BP_*BP \simeq BP_*[t_1,\ldots]$$ Denote the stack associated to $(V,W)$ by $\mathcal{M}$. This is an algebraic stack (\cite[4.1]{LMB},\cite{Nau}). Let $BP^{\wedge\bullet +1}$ denote the usual cosimplicial ring spectrum, then $\mathcal{M}$ is the stack associated to the simplicial affine scheme $\hbox{Spec}\,\pi_*BP^{\wedge\bullet +1}$. $$\mathcal{M} = \hbox{colim}_{\Delta^{op}}\hbox{Spec}\,\pi_*BP^{\wedge\bullet +1}$$ The colimit is taken in the $2$-category of stacks \cite[12.4]{LMB}. 

The algebraic stack $\mathcal{M}$ is closely related to the moduli stack of one-dimensional commutative formal groups. A formal group over a scheme $S$ is a commutative group object in the category of formal schemes over $S$ that is {\it fpqc}-locally isomorphic to $(\widehat{\mathbb{A}}^1,0)$ as schemes. Let's denote this stack by $X_{1,fg}$. Given a ring $R$, the groupoid $X_{1,fg}(R)$ comprises formal groups $G/R$ over $\hbox{Spec}\,R$ and their isomophisms.

The stack $X_{1,fg}$ carries a canonical line bundle $\omega$. For every $R$ we can construct the locally free rank one $R$-module $\omega_{G/R}$ of invariant $1$-forms of $G$ over $\hbox{Spec}\,R$, and its formation is compatible with base change and therefore defines a line bundle $\omega$ over $X$. 

The stack $\mathcal{M}= \mathcal{M}_{(V,W)}$ is a $\mathbb{G}_m$ torsor over $X_{1,fg}\otimes\mathbb{Z}_{(p)}$. Its points can be described as follows: For a $\mathbb{Z}_{(p)}$-algebra $R$, the groupoid $\mathcal{M}(R)$ consists of pairs $(G/R, \alpha:\omega_{G/R} \simeq R)$ of a formal group and a trivialization of $\omega$ section, and isomorphisms of $G$ that respect the trivializations.

\begin{definition}

Let $X$ be a $p$-local homotopy commutative ring spectrum.

Consider the cosimplicial ring spectrum $BP^{\wedge\bullet +1}\wedge X$ and the associated simplicial affine scheme $\hbox{Spec}\,\pi_*(BP^{\wedge\bullet +1}\wedge X)$. Define $$\mathcal{M}_X = \hbox{colim}_{\Delta^{op}}\,\hbox{Spec}\,\pi_*(BP^{\wedge\bullet +1}\wedge X)$$ By definition, $\mathcal{M}_X$ is an algebraic stack over $\hbox{Spec}\,\mathbb{Z}_{(p)}$.
\end{definition}

We now make clear the relation of $\mathcal{M}_X$ with formal groups. The map of cosimplicial ring spectra $BP^{\wedge\bullet +1} \rightarrow BP^{\wedge\bullet +1} \wedge X$ gives a map of algebraic stacks $\mathcal{M}_X \rightarrow \mathcal{M}$. There is an algebraic stack $M_X$ over $X_{1,fg}\otimes\mathbb{Z}_{(p)}$ along with a line bundle $\omega_X: M_X \rightarrow B\mathbb{G}_m$, so that $\mathcal{M}_X$ is a $\mathbb{G}_m$-torsor over $M_X$. The $R$-points of $\mathcal{M}_X$ can then be identified with pairs $(P\in M_X(R), \alpha:\omega_X(P/R) \simeq R)$ of objects of $M_X$ over $\hbox{Spec}\,R$ and trivialzations of their $\omega_X$ sections.

If $X$ is a homotopy $BP$-algebra then $\mathcal{M}_X$ is an affine scheme.



\begin{lemma}\label{smash}
If the map $\mathcal{M}_X \rightarrow \mathcal{M}$ is flat and $Y$ is another $p$-local homotopy commutative ring spectrum, then the stack associated to the smash product product $X\wedge Y$ can be identified with the pullback. 
$$\mathcal{M}_{X\wedge Y} \simeq \mathcal{M}_X \times_{\mathcal{M}} \mathcal{M}_Y$$
\end{lemma}

\begin{proof}
Consider the affine case first. Suppose $E$ and $F$ are $BP$-algebras and the map $\mathcal{M}_E \rightarrow \mathcal{M}$ is flat. This means that $E$ is a Landweber exact cohomology theory and there is an isomophism $E_*F \simeq E_*\otimes_{BP_*} BP_*BP \otimes_{BP_*}F_*.$ In terms of stacks this is,
$$\hbox{Spec}\,E_*F \simeq \hbox{Spec}\,E_* \times_{\mathcal{M}} \hbox{Spec}\,F_*. $$

The rest of the argument follows via descent. Suppose $\mathcal{X}$ and $\mathcal{Y}$ are algebraic stacks with maps to $\mathcal{M}$. Let $\mathcal{V}_{\bullet}$ and $\mathcal{W}_{\bullet}$ be simplicial affine schemes such that $\mathcal{X} = \hbox{colim}_{\Delta^{op}}\, \mathcal{V}_{\bullet}$ and $\mathcal{Y} \simeq \hbox{colim}_{\Delta^{op}}\,\mathcal{W}_{\bullet}$. Then the pullback can be given an atlas by the bisimplicial affine (since $\mathcal{M}$ is algebraic) scheme $\mathcal{X}\times_{\mathcal{M}}\mathcal{Y} \simeq \hbox{colim}_{\Delta^{op}\times \Delta^{op}}\, \mathcal{V}_{\bullet} \times_{\mathcal{M}}\mathcal{W}_{\bullet}$. 

If $\mathcal{M}_X \rightarrow \mathcal{X}$ is flat, it follows from the proof in the affine case that there is an equivalence of bisimplicial affine schemes $\hbox{Spec}\,\pi_*\,BP^{\wedge\bullet+1}\wedge X\times_{\mathcal{M}} \hbox{Spec}\,\pi_*\,BP^{\wedge\bullet+1}\wedge Y \simeq \hbox{Spec}\,\pi_*\,BP^{\wedge\bullet+1}\wedge X\wedge BP^{\wedge\bullet+1}\wedge Y$. The totalization of the bicosimplicial ring spectrum $BP^{\wedge\bullet+1}\wedge X\wedge BP^{\wedge\bullet+1}\wedge Y$ is $X\wedge Y$.

\end{proof}

The notion of height of a formal group gives a filtration of the moduli stack $X_{1,fg}$ and this canonically lifts to a filtration of the $\mathbb{G}_m$-torsor $\mathcal{M}$. One can give an explicit construction of the filtration. Let $I_n= (p,v_1,\ldots,v_{n-1})$ denote the invariant prime ideals of $V$. The associated substacks correspond to formal groups of height at least $n$.
$$\mathcal{M}^{\geq n} = \hbox{Spec}\,(V/I_n, V/I_n\otimes_V W \otimes_V V/I_n)$$

There is a filtration on $\mathcal{M}$ by closed substacks
$$\mathcal{M}= \mathcal{M}^{\geq 0} \supseteq \mathcal{M}^{\geq 1} \supseteq \ldots \supseteq\mathcal{M}^{\geq \infty}.$$

Let $\mathcal{U}^{n} = \mathcal{M} - \mathcal{M}^{\geq n}\,\,(0\leq n \leq \infty)$ be the open substack of $\mathcal{M}$ complementary to $\mathcal{M}^{\leq n}$. There is an ascending chain on open immersions $$\emptyset =\mathcal{U}^0 \subseteq \mathcal{U}^1 \subseteq \ldots \subseteq \mathcal{U}^{\infty} \subseteq \mathcal{M}.$$ Since for $0\leq n < \infty$, $I_n$ is finitely generated, the open immersion $\mathcal{U}^n \subseteq \mathcal{M}$ is quasi-compact and $\mathcal{U}^n$ is an algebraic stack. The $R$-points of $\mathcal{U}^n$ consist of formal groups of height at most $n-1$ over $\hbox{Spec}\,R$ along with trivializations of their corresponding cotangent bundles as before. The following corollary to \cite[Theorem 26]{Nau} gives an explicit atlas for $\mathcal{U}^{n+1}$.

\begin{prop}\label{En}
 Let $E(n)$ be the n-th Johnson Wilson spectrum with $E(n)_* = \mathbb{Z}_{(p)}[v_1,\ldots,v_n^{\pm 1}]$. Let $$(V_n,W_n):=(E(n)_*, E(n)_*\otimes_{BP_*} BP_*BP \otimes _{BP_*} E(n)_*)$$ be the Hopf algebroid induced from $(BP_*,BP_*BP)$ by the Landweber exact map $BP_*\rightarrow E(n)_*$. Then the Hopf algebroid $(V_n,W_n)$ is flat and its associated algebraic stack is $\mathcal{U}^{n+1}$. Equivalently, $$\hbox{colim}_{\Delta^{op}}\hbox{Spec}\,\pi_*E(n)^{\wedge\bullet +1} \simeq \mathcal{U}^{n+1}.$$
\end{prop} 

Lemma \ref{smash} gives us the following representation of the pullback.
$$\hbox{colim}_{\Delta^{op}}\hbox{Spec}\,\pi_*(E(n)^{\wedge\bullet +1}\wedge X) \simeq \mathcal{U}^{n+1} \times_{\mathcal{M}} \mathcal{M}_X$$

Let $\mathcal{M}^n = \mathcal{M}^{\geq n}[v_n^{-1}]$ denote the height $n$-layer. The stack $\mathcal{M}^n$ is contained inside $\mathcal{U}^{n+1}$ as a closed substack. Define the formal neighborhood of $\mathcal{M}^n$ by taking the completion of $\mathcal{U}^{n+1}$ at $\mathcal{M}^n$.

\begin{definition}
 $\widehat{\mathcal{M}}^n = (\mathcal{U}^{n+1})^{\wedge}_{I_{n}}$ 
\end{definition}

Let $$\eta:\hbox{Spec}\,\mathbb{F}_{p^n}[u^{\pm 1}] \rightarrow \mathcal{M}^n$$ be the Honda formal group of height $n$. This is a presentation for $\mathcal{M}^n$ and a pro-\'etale $\mathbb{S}_n \rtimes \hbox{Gal}(\mathbb{F}_{p^n}/\mathbb{F}_p)$-torsor. Let $\hbox{Def}(\Gamma_n, \mathbb{F}_{p^n})$ be the Lubin-Tate space of deformations of the Honda formal group $\Gamma_n$ over $\mathbb{F}_{p^n}$. Let $\hbox{Def}(\Gamma_n,\mathbb{F}_{p^n})[u^{\pm 1}]:= \hbox{Def}(\Gamma,\mathbb{F}_{p^n})\times \hbox{Spec}\,\mathbb{Z}[u^{\pm 1}]$. There is a map $$\hbox{Def}(\Gamma_n, \mathbb{F}_{p^n})[u^{\pm 1}] \rightarrow \widehat{\mathcal{M}}^n$$ which is a presentation and an pro-\'etale $\mathbb{S}_n \rtimes \hbox{Gal}(\mathbb{F}_{p^n}/\mathbb{F}_p)$-torsor. Furthermore there is a pullback of stacks,

$$\xymatrix{
\hbox{Spec}\,\mathbb{F}_{p^n}[u^{\pm 1}] \ar[r] \ar[d] &\hbox{Def}(\Gamma_n, \mathbb{F}_{p^n})[u^{\pm 1}] \ar[d] \\
\mathcal{M}^n \ar[r] &\widehat{\mathcal{M}}^n
}$$

Let $E_n$ denote the Hopkins-Miller spectrum associated to $\hbox{Def}(\Gamma_n,\mathbb{F}_{p^n})$. There is an isomorphism $\hbox{Def}(\Gamma_n,\mathbb{F}_{p^n})[u^{\pm 1}] \simeq \hbox{Spf}\,\pi_*E_n$. The diagonal, 

\begin{eqnarray*}
\hbox{Def}(\Gamma_n,\mathbb{F}_{p^n})\times_{\widehat{\mathcal{M}}^n} \hbox{Def}(\Gamma_n,\mathbb{F}_{p^n})[u^{\pm 1}] &\simeq& \hbox{Spf}\,E_{n*} \times (\mathbb{S}_n \rtimes Gal(\mathbb{F}_{p^n}/\mathbb{F}_p))\\ &\simeq& \hbox{Spf}\,\pi_*L_{K(n)}(E_n\wedge E_n).
\end{eqnarray*}

This produces an atlas for $\widehat{\mathcal{M}}^n$. 
\begin{prop}
$\hbox{colim}_{\Delta^{op}}\hbox{Spf}\,\pi_*L_{K(n)}E_n^{\wedge\bullet +1} \simeq \widehat{\mathcal{M}}^n$
\end{prop}

Lemma 2.1 gives us the following representation of the pullback.
$$\hbox{colim}_{\Delta^{op}}\hbox{Spf}\,\pi_*\,L_{K(n)}(E_n^{\wedge\bullet +1}\wedge X) \simeq \widehat{\mathcal{M}}^n \times_{\mathcal{M}} \mathcal{M}_X$$

For $X$ that is $E(n)$-local, the stack $\mathcal{M}_X$ can be assembled from the pieces $\widehat{\mathcal{M}}^k\times_{\mathcal{M}} \mathcal{M}_X$ for $0\leq k \leq n$. This follows from the observation that 
 
$$\xymatrix{
\widehat{\mathcal{M}}^n\times_{\mathcal{M}} \mathcal{U}^n \ar[r] \ar[d] &\mathcal{U}^n \ar[d]\\
\widehat{\mathcal{M}}^n  \ar[r] &\mathcal{U}^{n+1}
}$$ is a pushout square of stacks (see \cite[Theorem C]{Rydh} for reference on stack pushouts) and $\mathcal{M}_X \simeq \mathcal{U}^{n+1} \times_{\mathcal{M}} \mathcal{M}_X$ if $X$ is $E(n)$-local.

The remainder of this paper concerned with working out the structure of $\mathcal{M}_{ER(2)}$ near the height $2$ point.

\subsection{Elliptic curves} 

A Weierstrass curve over $R$ is the closure in $\mathbb{P}^2_R$ of the affine curve 
\begin{equation}\label{Weierstrass}
y^2 + a_1xy + a_3y = x^3 + a_2x^2 + a_4x + a_6
\end{equation}
over $R$. The curve is smooth if and only if $\Delta = \Delta(a_1,\ldots,a_6)$ is invertible in $R$. A strict isomorphism of Weierstrass curves in given by the change of coordinates $$x'= x+r,\,\, y' = y + sx +t.$$ The Weierstrass curves along with their coordinate changes form an algebraic stack $\mathcal{M}_{(A,\Gamma)}$ determined by the Hopf algebroid $(A, \Gamma)$ where $$A = \mathbb{Z}[a_1,a_2,a_3,a_4,a_6], \,\,\Gamma = A[s,r,t]$$ The Hopf algebroid structure maps are implicit in the definitions. (see \cite[section 3]{HM}).

We now explain how this stack $\mathcal{M}_{(A,\Gamma)}$ is associated to elliptic curves. Let $Ell$ denote the moduli stack of elliptic curves over $\hbox{Spec}\,\mathbb{Z}$.  A morphism $\hbox{Spec}(R) \rightarrow Ell$ classifies an elliptic curve $C \rightarrow R$, which is a smooth proper morphism whose geometric fibers are elliptic curves. Let $\overline{Ell}$ denote the compactified moduli stack classifying generalized elliptic curves. There exists a line bundle $\omega \rightarrow \overline{Ell}$ associated to the cotangent space at the identity section of a generalized elliptic curve. Given a smooth elliptic curve $C:\hbox{Spec}\,R \rightarrow Ell$, the set of sections $\Gamma(\hbox{Spec}\,R, \omega(C))$ is the set of invariant $1$-forms on $C$.

Let $\mathsf{Ell}$ be the $\mathbb{G}_m$-torsor over $Ell$ whose $R$-points are given by pairs $(C/R, \alpha: \omega(C) \simeq R)$. Here $C/R$ is an elliptic curve over $R$ and $\alpha$ is a choice of trivialization of the $R$-sections of $\omega$. 

Any generalized elliptic curve $C\rightarrow S$ admits a presentation in the Weierstrass normal form locally over $S$ in the flat topology. The identity element of the elliptic curve is identified with the unique point at infinity of the Weierstrass curve. This gives a map of stacks $\overline{\mathsf{Ell}} \rightarrow \mathcal{M}_{(A,\Gamma)}$ which is an equivalence on the substack of smooth elliptic curves, $\mathsf{Ell} \simeq \mathcal{M}_{(A[\Delta ^{-1}], \Gamma[\Delta^{-1}])}$.

There is a natural map $Ell \rightarrow X_{1,fg}$ classifying the formal group associated to an elliptic curve. This map lifts canonically to a map of the $\mathbb{G}_m$-torsors. $$\mathsf{Ell} \rightarrow \mathcal{M}$$

Consider the substack $\overline{\mathsf{Ell}}_p \rightarrow \overline{\mathsf{Ell}}$ which is the $p$-completion of $\overline{\mathsf{Ell}}$. Note that $\overline{\mathsf{Ell}}_p$ is a {\it formal} Deligne-Mumford stack. For any $p$-complete ring $R$ the map $\hbox{Spf}(R) \rightarrow \overline{\mathsf{Ell}}$ classifies an ind-system $C_m/\hbox{Spec}(R/p^n)$ of generalized elliptic curves.

Define $(\overline{\mathsf{Ell}})_{\mathbb{F}_p} = \overline{\mathsf{Ell}}\times_{\mathbb{Z}} \hbox{Spec}(\mathbb{F}_p)$. Let $(\mathsf{Ell}^{ord})_{\mathbb{F}_p} \subset (\overline{\mathsf{Ell}})_{\mathbb{F}_p}$ denote the locus of ordinary generalized elliptic curves in characteristic $p$, and let 
$$(\mathsf{Ell}^{ss})_{\mathbb{F}_p} = (\overline{\mathsf{Ell}})_{\mathbb{F}_p} - (\mathsf{Ell}^{ord})_{\mathbb{F}_p}$$ denote the locus of supersingular elliptic curves in characteristic $p$.

Consider the substack $$\mathsf{Ell}^{ss}_p \subset \overline{\mathsf{Ell}}_p$$ where $\mathsf{Ell}^{ss}_p$ is the completion of $\overline{\mathsf{Ell}}$ at $(\mathsf{Ell}^{ss})_{\mathbb{F}_p}$. 

Define the Hopf algebroid,

$$A'= A[\Delta^{-1}]^{\wedge}_{(p,a_1)},\,\, \Gamma'=A'\otimes_A\Gamma\otimes_A A'.$$

Then,

$$\mathsf{Ell}^{ss}_p = \mathcal{M}_{(A',\Gamma')}.$$

The following is a special case of the Serre-Tate theorem for abelian schemes \cite[Theorem 1.2.1]{Katz}. 

\begin{theorem}(Serre,Tate)
Let $R$ be a Noetherian ring with $p$ nilpotent and $I$ a nilpotent ideal in $R$, $R_0= R/I$. Let $Ell_p^{ss}(R)$ denote the category of supersingular elliptic curves in characteristic $p$ over $R$ and $D(R,R_0)$ denote the category of triples $$(C_0,G, \epsilon)$$ consisting of a supersingular elliptic curve $C_0$ over $R_0$, a formal group $G$ over $R$ and an isomorphism of formal groups over $R_0$, $\epsilon: G_0 \simeq C_0^{\wedge}$, where $G_0$ is reduction modulo $I$ of $G$. Then the functor $$Ell_p^{ss}(R) \rightarrow D(R,R_0)$$ $$ C \mapsto (C_0,C^{\wedge}, \hbox{natural}\,\,\,\epsilon)$$ is an equivalence of categories.

\end{theorem}

The following is implied by the Serre-Tate theorem.

\begin{prop}
There is an equivalence of formal Deligne-Mumford stacks,
$$\mathsf{Ell}^{ss}_p \simeq \widehat{\mathcal{M}}^2 \times_{\mathcal{M}} \overline{\mathsf{Ell}}.$$
\end{prop}

Let $E$ be a supersingular elliptic curve over a field $k$ of characteristic $p$ classifying a point $\eta:\hbox{Spec}\,k \rightarrow (\mathsf{Ell})_{\mathbb{F}_p}$. Serre-Tate theory gives an isomophism of the formal neighborhood of $\eta$ with the universal deformation space for the formal group $\widehat{E}$. Therefore there is an equivalence of deformation spaces $$\hbox{Def}(E,k) \simeq \hbox{Def}(\widehat{E},k).$$

\subsection{Level structures}
In this section we review the modular curves. Let $S$ be a scheme over $\mathbb{Z}[1/N]$ and $C$ a smooth elliptic curve over $S$. Let $C[N]$ denote the $N$-torsion points of $C$. The group scheme $C[N]$ is \'etale locally over $S$ isomorphic to the discrete group scheme $(\mathbb{Z}/N\mathbb{Z})^2$ over $C$. Let $Ell(N)$ denote the moduli stack of pairs $(C,\phi)$ where $C$ is a smooth elliptic curve and $\phi$ is a {\it full level-$n$ structure}, a choice of isomorphism $$\xymatrix{\phi: (\mathbb{Z}/N\mathbb{Z})^2 \ar[r]^{\simeq} &C[N]}.$$ Equivalently, the points of ${Ell}(N)$ are triples $(C,P,Q)$ where $P$ and $Q$ are a pair of sections $S \rightarrow C[N]$ that are, locally over $S$, linearly independent. 

Let ${Ell}_1(N)$ denote the moduli stack of pairs $(C,P)$ where $P$ is a primitive $N$-torsion point over $C$. Finally, let ${Ell}_0(N)$ denote the moduli stack of pairs $(C,H)$ where $H$ is a choice of a subgroup scheme $H\subset C[N]$ isomorphic to $\mathbb{Z}/N\mathbb{Z}$.

In this paper we'll work with stacks $\mathsf{Ell}(N)$, $\mathsf{Ell}_1(N)$ and $\mathsf{Ell}_0(N)$ which are $\mathbb{G}_m$-torsors over the modular curves $Ell(N)$, $Ell_1(N)$ and $Ell_0(N)$. The $R$-points of $\mathsf{Ell}(N)$ consists of triples $(C,\eta, (P,Q))$, where $C$ is an elliptic curve over $R$, $\eta$ is a choice of a nowhere vanishing invariant $1$-form of $C$ and $(P,Q)$ is a full level-$N$ structure described as before.

\begin{theorem}
(Deligne, Rapoport) (\cite{KM})The moduli stack $Ell(N)$ is a smooth affine scheme over $\hbox{Spec}\,\mathbb{Z}[1/N]$ for $N\geq 3$. For $N\geq 4$ the moduli stack $Ell_1(N)$ is a smooth affine scheme over $\hbox{Spec}\,\mathbb{Z}[1/N]$.
\end{theorem}

The maps forgetting level structures induce a diagram 
$$\xymatrix{
\mathsf{Ell}(N) \ar[d]_{\Gamma_1(N)} \ar@/_5pc/[ddd]_{GL(2,N)} \ar@/^2pc/[dd]^{\Gamma_0(N)}\\
\mathsf{Ell}_1(N) \ar[d]_{(\mathbb{Z}/N\mathbb{Z})^{\times}} \\
\mathsf{Ell}_0(N) \ar[d] \\
\mathsf{Ell}\times \hbox{Spec}\,\mathbb{Z}[1/N]
}$$

where all the arrows are finite \'etale and the labeled ones are Galois.

The Galois groups are defined as follows:
$$\Gamma_1(N) = \left\{ \left(\begin{array}{cc} 1 & * \\0 & * \end{array}\right ) \in GL(2,N)\right\}$$
$$\Gamma_0(N) = \left\{ \left(\begin{array}{cc} * & * \\0 & * \end{array}\right ) \in GL(2,N)\right\}$$

Let $\mathsf{Ell}(N)_p$, $\mathsf{Ell}_1(N)_p$ and $\mathsf{Ell}_0(N)_p$ denote the completions at $p$. Let $\mathsf{Ell}(N)^{ss}_p$ denote the pullback,
$$\xymatrix{
\mathsf{Ell}(N)^{ss}_p \ar[r] \ar[d] &\mathsf{Ell}(N)_p \ar[d] \\
\mathsf{Ell}^{ss}_p \ar[r] &\mathsf{Ell}_p
}$$

Since $\mathsf{Ell}(N)_p$ is formal affine (assuming $N\geq 3$) and the right vertical arrow is an \'etale $GL(2,N)$-torsor, Serre-Tate theory implies there is an equivalence $$\mathsf{Ell}(N)_p^{ss} \simeq \coprod_i\hbox{Spf}\,W(k_i)[[u_1]][u^{\pm 1}]$$ for a finite set of fields $k_i$ (depending on $N$).

\subsection{Restatement of the main theorem}
Given a formal Deligne-Mumford stack $\mathcal{S}$ over $\mathbb{Z}_p$, define its maximal unramified cover $\mathcal{S}^{nr}$ to be the pullback $\mathcal{S}\times_{\mathbb{Z}_p}\hbox{Spf}\,W(\overline{\mathbb{F}}_p)$. The map $\mathcal{S}^{nr} \rightarrow \mathcal{S}$ is an pro-\'etale cover with Galois group $\hbox{Gal}\,(\overline{\mathbb{F}}_p/\mathbb{F}_p) = \hat{\mathbb{Z}}$. Our main Theorem 1.1 can be restated in the following way.

\begin{theorem}\label{restate}
There is a map of stacks over $\widehat{\mathcal{M}}^2$, $${\widehat{\mathcal{M}}^2 \times_{\mathcal{M}} \mathcal{M}_{ER(2)}}^{nr} \rightarrow \mathsf{Ell}_2^{ss}.$$ The map over $\widehat{\mathcal{M}}^2$ induced by forgetting level structure $$\left(\mathsf{Ell}_1(3)^{ss}_2\right)^{nr} \rightarrow \left(\mathsf{Ell}_0(3)^{ss}_2\right)^{nr}$$ is equivalent to the map $${\widehat{\mathcal{M}}^2 \times_{\mathcal{M}} \mathcal{M}_{E(2)}}^{nr} \rightarrow {\widehat{\mathcal{M}}^2 \times_{\mathcal{M}}\mathcal{M}_{ER(2)}}^{nr}$$ over $\widehat{\mathcal{M}}^2$ induced by the inclusion of fixed points $ER(2) \rightarrow E(2)$.

\end{theorem}

\section{Real Johnson-Wilson theory from modular curves}

\subsection{Level $3$-structures at the prime $2$}

Consider the supersingular elliptic curve
\begin{equation}\label{ss}
C:x^3+y^2+y=0 \in \mathbb{P}^2_{\mathbb{F}_4}.
\end{equation}

The automorphism group $G_{24} = \hbox{Aut}_{\mathbb{F}_4}(C)$ is the group of units in the maximal order of a rational quarternion algebra $\mathbb{Q}\{i,j,k\}$. It's isomorphic to the binary tetrahedral group $\widehat{A}_4 = Q_8 \rtimes C_3$ of order $24$, where $C_3$ acts on $Q_8$ by conjugation. It contains the quarternion group $Q_8=\{\pm 1,\pm i, \pm j \pm k\}$ and $16$ other elements $(\pm 1\pm i \pm j\pm k)/2$. 

Let $C^{\wedge}$ be the completion of $C$ at the identity section. $C^{\wedge}$ is a formal group of height $2$ over $\mathbb{F}_4$. The automorphism group $\hbox{Aut}(C^{\wedge}) = \mathcal{O}^{\times}_{D_{1/2},\mathbb{Q}_2}$ is the group of units in the maximal order of the $2$-adic quaternion algebra $$D_{\frac{1}{2},\mathbb{Q}_2} = \mathbb{Q}_2\{i,j,k\}.$$ Abstractly this is the completion of the Hurwitz lattice $\mathbb{Z}(\pm1,\pm i,\pm j, \pm k) \coprod (\pm 1\pm i\pm j \pm k)/2$ at the ideal $(2)$. Notice that $G_{24}$ is the maximal finite subgroup. 

\begin{prop}
The map $\hbox{Spec}\,\mathbb{F}_4 \rightarrow (Ell^{ss})_{\mathbb{F}_2}$ classifying $C$ is a presentation and an \'etale $G_{24} \rtimes \hbox{Gal}(\mathbb{F}_4/\mathbb{F}_2)$-torsor.
\end{prop}

It follows that the map $\hbox{Def}(C,\mathbb{F}_4) \rightarrow Ell^{ss}_2$ classifying the universal deformation of $C$ is an \'etale $G_{24} \rtimes \hbox{Gal}(\mathbb{F}_4/\mathbb{F}_2)$-torsor. By Serre-Tate theory there is an isomorphism of deformation spaces $$\hbox{Def}(C,\mathbb{F}_4) \simeq \hbox{Spf}\,W(\mathbb{F}_4)[[a_1]].$$ The curve $C$ lifts to $W(\mathbb{F}_4)[[a_1]]$ as $\widetilde{C}:y^2+ a_1xy + y = x^3$, which is the universal deformation curve over $\hbox{Def}(C,\mathbb{F}_4)$. $\widetilde{C}$ lifts further to 
\begin{equation}\label{ssdeform}
y^2+a_1uxy +u^3y =x^3
\end{equation}
over $\hbox{Spf}\,W(F_4)[[a_1]][u^{\pm 1}]$. We shall call this $\widetilde{C}$ from now.

\begin{prop} \cite[Theorem 3.1]{HM} The map $$\hbox{Spf}\,W(\mathbb{F}_4)[[a_1]][u^{\pm 1}] \rightarrow \mathsf{Ell}_2^{ss} \simeq \mathcal{M}_{(A',\Gamma')}$$ classifying the curve $$\widetilde{C}:y^2 + a_1uxy + u^3y = x^3$$ is a presentation and an \'etale $G_{24} \rtimes \hbox{Gal}(\mathbb{F}_4/\mathbb{F}_2)$-torsor.
\end{prop}

This is a restatement of $$L_{K(2)}TMF = EO_2.$$

The level $3$-structures on elliptic curves and their associated moduli stacks are related by finite \'etale morphisms
$$\xymatrix{
\mathsf{Ell}(3) \ar[r]^6 &\mathsf{Ell}_1(3) \ar[r]^2 &\mathsf{Ell}_0(3) \ar[rr]^4 &&\mathsf{Ell}\times\hbox{Spec}\,\mathbb{Z}[1/3]
}$$ of degrees 6, 2 and 4. The related modular groups are as follows:

$$\Gamma_1(3) = \left\{ \left(\begin{array}{cc} 1 & * \\0 & * \end{array}\right ) \in GL(2,3)\right\} \simeq C_3 \ltimes C_2$$
$$\Gamma_0(3) = \left\{ \left(\begin{array}{cc} * & * \\0 & * \end{array}\right ) \in GL(2,3)\right\} \simeq C_6 \ltimes C_2$$

The point $(0,0)$ on the Weierstrass normal form of a smooth elliptic curve is a point of order $3$ if it is of the form $$C(a_1,a_3): y^2 + a_1xy + a_3y = x^3, \,\, a_3 \neq 0.$$ The curve $C(a_1,a_3)$ comes with the invariant $1$-form $$\eta = \frac{dx}{2y+a_1x+a_3} = \frac{dy}{3x^2-a_1y}.$$ 

Furthermore, if $(C',P,\eta')$ is a smooth elliptic curve with a point $P$ of order $3$ and an invariant $1$-form $\eta'$ then there exists a {\it unique} isomorphism of the triple (see \cite[Prop. 3.2]{MR}) $$(C(a_1,a_3), (0,0), \eta) \simeq (C',P,\eta'). $$

This gives us the following equivalence of stacks. Notation: $\Delta = a_3^3(a_1^3 - 27a_3)$.

$$\mathsf{Ell}_1(3) = \hbox{Spec}\,\mathbb{Z}[1/3][a_1,a_3^{\pm 1},\Delta^{-1}]$$
$$\mathsf{Ell}_1(3)_2^{ss} = \hbox{Spf}\,\mathbb{Z}_2[[a_1]][a_3^{\pm 1}]$$

\begin{prop}\label{TMF13}
The map $$\hbox{Spf}\,W(\mathbb{F}_4)[[a_1]][u^{\pm 1}] \rightarrow \mathsf{Ell}_1(3)_2^{ss}$$ classifying the point $(\widetilde{C}, (0,0), \frac{dx}{2y+a_1ux+u^3})$ is a presentation and an \'etale $C_3\rtimes \hbox{Gal}(\mathbb{F}_4/\mathbb{F}_2)$-torsor. The embedding $C_3 \subset \mathbb{S}_2$ is induced by lifting the $2$-adic Teichm\"uller character $\mathbb{F}_4^{\times} \rightarrow W(\mathbb{F}_4)^{\times}$.
\end{prop}

There is a $C_2$-\'etale map forgetting level structure, $$\mathsf{Ell}_1(3) \rightarrow \mathsf{Ell}_0(3)$$ $$ (C,P,\eta) \mapsto (C,\langle P \rangle, \eta).$$ The group $C_2$ acts on $\mathsf{Ell}_1(3)$ by $(C,P,\eta) \mapsto (C,-P,\eta)$.

\begin{prop}\label{TMF03}
The map $$\hbox{Spf}\,W(\mathbb{F}_4)[[a_1]][u^{\pm 1}] \rightarrow \mathsf{Ell}_0(3)_2^{ss}$$ classifying the point $(\widetilde{C}, \langle(0,0)\rangle, \frac{dx}{2y+a_1ux+u^3})$ is a presentation and an \'etale $C_2\times C_3 \rtimes \hbox{Gal}(\mathbb{F}_4/\mathbb{F}_2)$-torsor. The group $C_2$ is the normal subgroup of $Q_8\rtimes C_3$ generated by the formal inverse $[-1]_{C^{\wedge}}$.
\end{prop}

The following diagram shows how the various modular curves at the supersingular locus in characterstic $2$ are related. All the arrows are finite \'etale and the labeled ones are Galois. Notation: $Gal = \hbox{Gal}(\mathbb{F}_4/\mathbb{F}_2)$.
$$
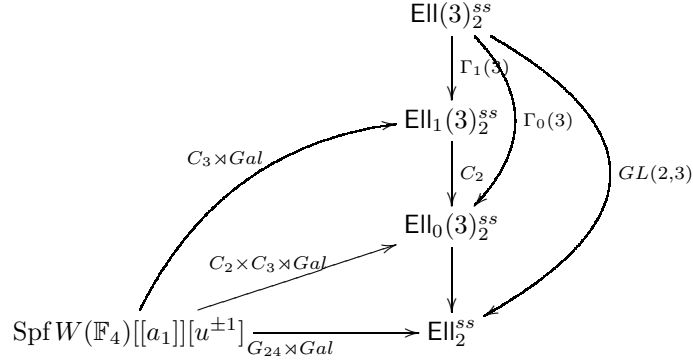
\xymatrix{
&&\mathsf{Ell}(3)_2^{ss} \ar[d]^{\Gamma_1(3)} \ar@/^5pc/[ddd]^{GL(2,3)} \ar@/^2pc/[dd]^{\Gamma_0(3)} \\
&&\mathsf{Ell}_1(3)_2^{ss} \ar[d]^{C_2} \\
&&\mathsf{Ell}_0(3)_2^{ss} \ar[d] \\
\hbox{Spf}\,W(\mathbb{F}_4)[[a_1]][u^{\pm 1}] \ar[rr]_{G_{24}\rtimes Gal} \ar[urr]^{C_2\times C_3 \rtimes Gal} \ar@/^2pc/[uurr]^{C_3\rtimes Gal} &&\mathsf{Ell}_2^{ss} 
}$$

As a consequence there is a diagram of $K(2)$-local elliptic spectra.

$$\xymatrix{
L_{K(2)}TMF_1(3)  \ar[r]^{\simeq} &E_2^{h(C_3\rtimes Gal)}\\
L_{K(2)}TMF_0(3) \ar[u] \ar[r]^{\simeq} &E_2^{h(C_2 \times C_3 \rtimes Gal)}\ar[u]\\
L_{K(2)}TMF \ar[u] \ar[r]^{\simeq} &E_2^{h(\hat{A}_4 \rtimes Gal)} \ar[u]
}$$

Here $E_2$ is the Hopkins-Miller Morava $E$-theory spectrum associated to the deformation space $\hbox{Def}(C,\mathbb{F}_4)$.

\subsection{The structure of $\mathcal{M}_{E(2)}$ near the height $2$ point}

Let the point $f:\hbox{Spec}\,\mathbb{F}_4 \rightarrow \mathcal{U}^3$ classify the formal group $C^{\wedge}$ associated to the supersingular elliptic curve $C:y^2+y+x^3=0$. Let $\widetilde{C}$ denote the universal deformation of the curve $C$ over $\hbox{Spf}\,W(\mathbb{F}_4)[[a_1]][u^{\pm 1}]$ as in equation (\ref{ssdeform}).

Since $\mathcal{U}^3$ denotes the moduli stack of formal groups of height at most $2$ we know from Prop. \ref{En} there is a faithfully flat presentation $$\hbox{Spec}\,\mathbb{Z}_{(2)}[v_1,v_2^{\pm 1}] \rightarrow \mathcal{U}^3.$$ 

As a consequence there exists an extension $K$ of $\mathbb{F}_4$ such that $$f: \hbox{Spf}\,W(K)[[a_1]][u^{\pm 1}] \rightarrow \hbox{Spf}\,W(\mathbb{F}_4)[[a_1]][u^{\pm 1}]$$ is an fpqc cover and the pullback $f^*(\widetilde{C}^{\wedge})$ is isomorphic to a $2$-typical formal group law $\psi: \hbox{Spf}\,W(K)[[a_1]][u^{\pm 1}] \rightarrow \hbox{Spec}\,\mathbb{Z}_{(2)}[v_1,v_2^{\pm 1}]$. $$\alpha: f^*(\widetilde{C}^{\wedge}) \simeq \psi$$

Consider the map
$$\xymatrix{
\hbox{Spf}\,W(K)[[a_1]][u^{\pm 1}] \ar[rr]^{(f^*(\widetilde{C}^{\wedge}), \alpha, \psi)} &&\widehat{\mathcal{M}}^2\times_{\mathcal{M}} \mathcal{M}_{E(2)}}$$ classifying the formal group $f^*(\widetilde{C}^{\wedge})$, $\psi$ and the isomorphism $\alpha$. This map is a surjection and an \'etale presentation. 

The Teichm\"uller character map $\mathbb{F}_4^{\times} \rightarrow \mathbb{Z}_2^{\times}$ lifts to give an embedding $\mathbb{F}_4^{\times} \subset \hbox{Aut}(f^*(C^{\wedge}))$. For a $\omega \in \mathbb{F}_4^{\times}$ the linear formal power series $g(t) = \omega t$ gives an element of $\mathbb{S}_2$. Since the action of $\hbox{Aut}(f^*(C^{\wedge}))$ extends to an action on the deformation space $\hbox{Spf}\,W(K)[[a_1]][u^{\pm 1}]$, we obtain an action of $g$ on the ring $W(K)[[a_1]][u^{\pm 1}]$ by $$g(u) = \omega u$$ $$g(ua_1) = \omega^2ua_1$$ which leaves $v_1$ and $v_2$ invariant. Moreover these are the only elements of $\mathbb{S}_2$ which acts invariantly on $v_1$ and $v_2$. Since $\widehat{\mathcal{M}}^2 \times_{\mathcal{M}} \hbox{Spec}\,\mathbb{Z}_{(2)}[v_1,v_2^{\pm 1}] \simeq \hbox{Spf}\,\mathbb{Z}_2[[v_1]][v_2^{\pm 1}]$ we can identify the pullback in the following diagram of stacks
$$\xymatrix{
\hbox{Spf}\,W(K)[[a_1]][u^{\pm 1}] \times (C_3 \rtimes \hbox{Gal}(K,\mathbb{F}_2)) \ar[d] \ar[r] &\hbox{Spf}\,W(K)[[a_1]][u^{\pm 1}] \ar[d]\\
\hbox{Spf}\,W(K)[[a_1]][u^{\pm 1}] \ar[r] &\widehat{\mathcal{M}}^2\times_{\mathcal{M}} \mathcal{M}_{E(2)}}$$ 
 
Therefore, the map $\hbox{Spf}\,W(K)[[a_1]][u^{\pm 1}] \rightarrow \widehat{\mathcal{M}}^2 \times_{\mathcal{M}} \hbox{Spec}\,\mathbb{Z}_{(2)}[v_1,v_2^{\pm 1}]$ is an \'etale $C_3 \rtimes \hbox{Gal}(K,\mathbb{F}_2)$-torsor. The conjugation action of $\hbox{Gal}(K/\mathbb{F}_2)$ on $C_3$ factors through the quotient $\hbox{Gal}(K/\mathbb{F}_2) \rightarrow \hbox{Gal}(\mathbb{F}_4/\mathbb{F}_2)$ since every automorphism of $\widetilde{C}^{\wedge}$ is already defined over $\mathbb{F}_4$. In terms of $K(2)$-localization, $$L_{K(2)}E(2) = E_2^{hH}$$ where $H = \mathbb{F}_4^{\times} \rtimes \hbox{Gal}(K/\mathbb{F}_2)$ and $E_2$ is the Hopkins-Miller spectrum associated to $\hbox{Def}(C^{\wedge},K)$.

More generally, we can use the maximal unramified extension $E_2^{nr}$, which is the Hopkins-Miller spectrum associated to the deformation space $\hbox{Def}(C,\overline{\mathbb{F}}_2)$. The extended Morava stabilizer group in this case is $\mathbb{S}_2 \rtimes \hbox{Gal}(\overline{\mathbb{F}}_2/\mathbb{F}_2) = \mathbb{S}_2 \rtimes \hat{\mathbb{Z}}$. 

\begin{prop}\label{E2}
There is a map of stacks $\hbox{Def}(C,\overline{\mathbb{F}}_2) \rightarrow \widehat{\mathcal{M}}^2 \times_{\mathcal{M}} \mathcal{M}_{E(2)}$ which is a surjection and an \'etale $C_3 \rtimes \hat{\mathbb{Z}}$-torsor. 
\end{prop}

The maximal unramified cover of the deformation space $\hbox{Def}(C,\mathbb{F}_4)$ can be identified with $\hbox{Def}(C,\overline{\mathbb{F}}_2)$. There is a natural map of stacks (we use the notation \quo for quotient stacks)
\begin{equation}\label{mapnr}
\hbox{Def}(C,\overline{\mathbb{F}}_2) \rightarrow \hbox{Def}(C,\mathbb{F}_4)\quo \hbox{Gal}(\mathbb{F}_4/\mathbb{F}_2),
\end{equation}
which realizes as a $K(n)$-local $\hat{\mathbb{Z}}$-Galois extension of $S$-algebras $$E_2^{Gal} \rightarrow E_2^{nr}.$$ The $\hat{\mathbb{Z}}$-extension is obtained by adjoining all the roots of unity of order prime to $p$ to the $p$-complete spectrum $E_2^{Gal}$ (\cite[5.4.6]{Rognes}).

The map (\ref{mapnr}) factors through the quotient $$\hbox{Def}(C,\overline{\mathbb{F}}_2)\quo C_3 \rightarrow \hbox{Def}(C, \mathbb{F}_4)\quo \widehat{A}_4 \rtimes \hbox{Gal}(\mathbb{F}_4/\mathbb{F}_2).$$ 

In terms of $K(2)$-localization there is a map 

\begin{equation}\label{map1}
L_{K(2)}TMF \simeq (E_2)^{\widehat{A}_4 \rtimes Gal} \rightarrow (E_2^{nr})^{hC_3} \simeq L_{K(2)}E(2)(\mu_{\infty,2}).
\end{equation} 
  
Combining Prop. \ref{TMF13} and Prop. \ref{E2} proves the following.
\begin{prop}\label{TMF1E2}
There are equivalences of stacks over $\widehat{\mathcal{M}}^2$.
$$\hbox{Def}(C,\overline{\mathbb{F}}_2)\quo C_3 \simeq {\widehat{\mathcal{M}}^2 \times_{\mathcal{M}} \mathcal{M}_{E(2)} }^{nr} \simeq \left( {\mathsf{Ell}_1(3)}_2^{ss} \right)^{nr}.$$
\end{prop}

\subsection{The structure of $\mathcal{M}_{ER(2)}$ near the height $2$ point}

\begin{prop}
The map of stacks $$\widehat{\mathcal{M}}^2 \times_{\mathcal{M}} \mathcal{M}_{E(2)} \rightarrow \widehat{\mathcal{M}}^2 \times_{\mathcal{M}} \mathcal{M}_{ER(2)}$$ is a surjection and an \'etale $C_2$-torsor.
\end{prop}

\begin{proof}
Combining Averett's formula (\ref{Averett}) and \cite[Theorem 5.4.4]{Rognes} we see that the map of ring spectra $L_{K(2)}ER(2) \rightarrow L_{K(2)}E(2)$ is a $K(2)$-local $C_2$-Galois extension of $2$-complete commutative $S$-algebras.
\end{proof}

Given the description of $\widehat{\mathcal{M}}^2 \times_{\mathcal{M}} \mathcal{M}_{E(2)}$ as the quotient stack $\hbox{Def}(C,\overline{\mathbb{F}}_2)\quo C_3\rtimes\hat{\mathbb{Z}}$, the action of $C_2$ must come from the inverse map $[-1]_C$ on the curve $C$. 


The induced action on the deformation space is as follows: $$\xymatrix{[-1]_C: W(\mathbb{F}_4)[[a_1]][u^{\pm 1}] \ar[r] & W(\mathbb{F}_4)[[a_1]][u^{\pm 1}]}$$ $$[-1]_C(u)=-u$$ $$[-1]_C(ua_1)=ua_1$$

\begin{prop}\label{ER2}
There is a map of stacks $\hbox{Def}(C,\overline{\mathbb{F}}_2) \rightarrow \widehat{\mathcal{M}}^2 \times_{\mathcal{M}} \mathcal{M}_{ER(2)}$ which is an \'etale surjection and a $C_2 \times C_3 \rtimes \hat{\mathbb{Z}}$-torsor. The $C_2$-action on $\hbox{Def}(C,\overline{\mathbb{F}}_2)$ comes from the inverse $[-1]_C$.
\end{prop}

This produces a map $$\hbox{Def}(C,\overline{\mathbb{F}}_2)\quo C_2 \times C_3 \rightarrow \hbox{Def}(C, \mathbb{F}_4)\quo \widehat{A}_4 \rtimes \hbox{Gal}(\mathbb{F}_4/\mathbb{F}_2).$$

In terms of $K(2)$-localization the map of spectra (\ref{map1}) factors through the real part.

\begin{equation}\label{map2}
L_{K(2)}TMF \rightarrow (E_2^{nr})^{h(C_2\times C_3)} \simeq L_{K(2)}ER(2)(\mu_{\infty,2}) \rightarrow L_{K(2)}E(2)(\mu_{\infty,2}).
\end{equation}

This proves the first part of Theorem \ref{restate}. Combining Prop. \ref{TMF03} with Prop. \ref{ER2} proves the following which together with Prop. \ref{TMF1E2} proves the rest of Theorem \ref{restate}.

\begin{prop}\label{TMF0ER2}
There are equivalences of stacks over $\widehat{\mathcal{M}}^2$,
$$\hbox{Def}(C,\overline{\mathbb{F}}_2)\quo C_2 \times C_3 \simeq {\widehat{\mathcal{M}}^2 \times_{\mathcal{M}}\mathcal{M}_{ER(2)}}^{nr} \simeq \left({\mathsf{Ell}_0(3)}^{ss}_2\right)^{nr}.$$
\end{prop}

\end{document}